\documentclass[11pt,twoside]{amsart}
\usepackage[dvips]{graphicx}
\usepackage{amssymb}
\usepackage{latexsym}
\usepackage{xypic}

\newtheorem{theorem}{Theorem}[section]
\newtheorem{lemma}[theorem]{Lemma}
\newtheorem{proposition}[theorem]{Proposition}

\newtheorem{corollary}[theorem]{Corollary}

\theoremstyle{definition}
\newtheorem{definition}[theorem]{Definition}

\theoremstyle{remark}

\newenvironment{definition-proposition}{\begin{def-prop} \em}{\end{def-prop}}

\def\P{\mathbb P_{\mathbb C}}

\numberwithin{equation}{section}

\newcommand{\C}{\mathbb{C}}

\begin{document}

\title{  \sc{ Two  dimensional Complex Kleinian Groups With  Four Complex  Lines in General Position in its Limit Set }}
\author{W. Barrera, A. Cano \& J. P. Navarrete }
\address{Waldemar Barrera: Universidad Aut\'onoma de Yucat\'an Facultad de Matem\'aticas, Anillo Perif\'erico
Norte Tablaje Cat 13615 Chuburn\'a Hidalgo, M\'erida Yucat\'an. M\'exico.\\
Angel Cano:   Instituto de Matem\'atica Pura e Aplicada (IMPA), Rio de Janeiro, Brazil.\\
Juan Pablo Navarrete:  Universidad Aut\'onoma de Yucat\'an Facultad de Matem\'aticas, Anillo Perif\'erico
Norte Tablaje Cat 13615 Chuburn\'a Hidalgo, M\'erida Yucat\'an. M\'exico.\\ }
\email{bvargas@uady.mx, angel@impa.br, jp.navarrete@uady.mx}

\thanks{Research partially supported by grants from CNPq}




\begin{abstract}
In this article we provide an algebraic characterization of those groups of $PSL(3,\Bbb{C})$ whose limit set in the  Kulkarni sense has, exactly, four lines in general position. Also we show that, for this class of groups,  the equicontinuity set of the group is the largest open set where the group  acts discontinuously and agrees  with the discontinuity set of the group. 	
\end{abstract}

\maketitle

\section*{ Introduction}

In a recent article, see \cite{bcn},  we have proven that for a complex Kleinian group without proper invariant  subspaces and  "enough" lines in the Kulkarni's limit set, it holds that its discontinuity  set  agrees with the equicontinuity  set of the group, is the largest open set where the group acts discontinuously and is a holomorphy domain. Such result enable  us to understand the relationship, in the two dimensional case and  for a "large class" of groups, bettwen the different notions of limit sets which are usually studied as well as its geometry, see \cite{csn}. This article is  a step to extent the results in \cite{bcn} to the case when the groups has invariant subspaces and "enough" lines in the limit set. More precisely we prove:\\
 
\begin{theorem} \label{t:main}
 Let $\Gamma\subset PSL(3,\C)$ be a discrete group. The limit set, in the Kulkarni sense, of $\Gamma$ has exactly four lines in general position   if and only if   $\Gamma$ has a hyperbolic toral group, see section \ref{s:toral}, whose index is at most $8$.
\end{theorem}

\begin{theorem}
Let $\Gamma\subset PSL(3,\Bbb{C})$ be a toral group. Thus the discontinuity set in the  Kulkarni sense  agrees with the equicontinuity region and is given by:
\[
\Omega(\Gamma)=\bigcup_{\epsilon_1,\epsilon_2=\pm 1}\Bbb{H}^{\epsilon_1}\times \Bbb{H}^{\epsilon_2}, 
\]
where $\Bbb{H}^{+1}$ and $\Bbb{H}^{-1}$ are the    upper half  and lower half plane.  
Moreover $\Omega(\Gamma)$ is the largest open set on which $\Gamma$ acts properly discontinuously.
\end{theorem}

This article is  organized as follows: in section \ref{s:prel} we introduce some terms  and  notations which will be used along the text. In section  \ref{s:toral} we construct some examples of groups with four lines in general position.  Finally in section \ref{s:4lines}, we present the proof of theorem \ref{t:main}.\\

The authors are grateful to Professor Jos\'e Seade for stimulating conversations. Part of this research was done while the authors were visiting the
IMATE-UNAM campus Cuernavaca and the FMAT of the UADY. Also,
during this time, the second author was in a postdoctoral year at IMPA,
and they are grateful to these institutions and its people, for their support
and hospitality.

\section{Preliminaries and Notations} \label{s:prel}

\subsection{Projective Geometry}
We recall that the complex projective plane
$\mathbb{P}^2_{\mathbb{C}}$ is $$\Bbb{P}^2_\Bbb{C}:=(\mathbb{C}^{3}\setminus \{0\})/\mathbb{C}^*
,$$
 where $\mathbb{C}^*$ acts on $\mathbb{C}^3\setminus\{0\}$ by the usual scalar
 multiplication. This  is   a  compact connected  complex $2$-dimensional manifold.
Let $[\mbox{ }]:\mathbb{C}^{3}\setminus\{0\}\rightarrow
\mathbb{P}^{2}_{\mathbb{C}}$ be   the quotient map. If
$\beta=\{e_1,e_2,e_3\}$ is the standard basis of $\mathbb{C}^3$, we
will write $[e_j]=e_j$ and if $w=(w_1,w_2,w_3)\in
\mathbb{C}^3\setminus\{0\}$ then we will  write   $[w]=[w_1:w_2:w_3]$.
Also, $\ell\subset \mathbb{P}^2_{\mathbb{C}}$ is said to be a
complex line if $[\ell]^{-1}\cup \{0\}$ is a complex linear
subspace of dimension $2$. Given  $p,q\in
\mathbb{P}^2_{\mathbb{C}}$ distinct points,     there is a unique
complex line passing through  $p$ and $q$, such line will be
denoted by $\overleftrightarrow{p,q}$.\\

Consider the action of $\mathbb{Z}_{3}$ (viewed as the cubic roots of
the unity) on  $SL(3,\mathbb{C})$ given by the usual scalar
multiplication, then
$$PSL(3,\mathbb{C})=SL(3,\mathbb{C})/\mathbb{Z}_{3}$$ is a Lie group
whose elements are called projective transformations.  Let
$[[\mbox{  }]]:SL(3,\mathbb{C})\rightarrow PSL(3,\mathbb{C})$ be   the
quotient map,   $\gamma\in PSL(3,\mathbb{C})$ and  $\widetilde\gamma\in
GL(3,\mathbb{C})$, we will say that  $\tilde\gamma$  is a  lift of
$\gamma$ if there is a cubic root $\tau$ of  $Det(\gamma)$ such that   $[[\tau \widetilde\gamma]]=\gamma$, also, we will use the notation $(\gamma_{ij})$ to denote elements  in $SL(3,\Bbb{C})$. One can show that
$ PSL(3,\mathbb{C})$ is a Lie group  that acts  transitively,
effectively and by biholomorphisms  on $\mathbb{P}^2_{\mathbb{C}}$
by $[[\gamma]]([w])=[\gamma(w)]$, where $w\in
\mathbb{C}^3\setminus\{0\}$ and    $\gamma\in SL_3(\mathbb{C})$.

\subsection{Complex Kleinian Groups}
 Let $\Gamma\subset   PSL(3,\mathbb{C})$ be a subgroup. We  define
 (following Kulkarni, see  \cite{kulkarni}): the set 
 $L_0(\Gamma)$  as the closure  of  the points in
$\mathbb{P}^2_{\mathbb{C}}$ with infinite isotropy group. The set $L_1(\Gamma)$ as the closure of the set  of cluster points  of
$\Gamma z$  where  $z$ runs  over  $\mathbb{P}^2_{\mathbb{C}}\setminus
L_0(\Gamma)$. Recall that $q$ is a cluster point  for  $\Gamma K$,
where $K\subset \mathbb{P}^2_{\mathbb{C}}$ is a non-empty set, if there is a sequence
$(k_m)_{m\in\mathbb{N}}\subset K$ and a sequence of distinct elements
$(\gamma_m)_{m\in\mathbb{N}}\subset \Gamma$ such that
$\gamma_m(k_m)\xymatrix{ \ar[r]_{m \rightarrow  \infty}&} q$. The set   $L_2(\Gamma)$ as  the closure of cluster  points of $\Gamma
K$  where $K$ runs  over all  the compact sets in
$\mathbb{P}^2_{\mathbb{C}}\setminus (L_0(\Gamma) \cup L_1(\Gamma))$. The  \textit{Limit Set in the sense of Kulkarni} for $\Gamma$  is
defined as:  $$\Lambda (\Gamma) = L_0(\Gamma) \cup
L_1(\Gamma) \cup L_2(\Gamma).$$  The \textit{Discontinuity
Region in the sense of Kulkarni} of $\Gamma$ is defined as:
$$\Omega(\Gamma) = \mathbb{P}^2_{\mathbb{C}}\setminus
\Lambda(\Gamma).$$
 We will say  that $\Gamma$ is a \textit{Complex  Kleinian Group}
 if $\Omega(\Gamma)\neq \emptyset$.

\begin{lemma} \label{l:control} ( See \cite{cano}) Let   $ \Gamma\subset PSL_3(\mathbb{C})$ be a subgroup, $p\in \mathbb{P}^2_{\mathbb{C}}$
such that  $\Gamma p=p$ and $\ell$ a complex line not containing $p$.
Define $\Pi=\Pi_{p,\ell}:\Gamma\longrightarrow  Bihol(\ell) $
given by $\Pi(g)(x)=\pi(g(x))$ where
$\pi=\pi_{p,\ell}:\mathbb{P}^2_{\mathbb{C}}- \{p\}\longrightarrow
\ell$ is given by $\pi(x)=\overleftrightarrow{x,p}\cap \ell$, then:
\begin{enumerate}
\item \label{i:con1} $\pi$ is a holomorphic  function. \item
\label{i:con2} $\Pi$   is a group morphism. \item \label{i:con3}
If   $ Ker(\Pi)$ is finite and $\Pi(\Gamma)$ is discrete, then
$\Gamma$ acts   discontinuously on $\Omega=(\bigcup_{z\in
\Omega (\Pi (\Gamma))}\overleftrightarrow{z,p})- \{p\}$. Here
$\Omega(\Pi(\Gamma))$ denotes the discontinuity set of
$\Pi(\Gamma)$.
 \item \label{i:con4} If $\Gamma$ is discrete, $\Pi(\Gamma)$
is non-discrete and  $\ell$ is invariant, then $\Gamma$ acts
discontinuously on $\Omega= \bigcup_{z\in
Eq(\Pi(\Gamma))} \overleftrightarrow{z,p}-(\ell\cup\{p\} )$.
\end{enumerate}

\end{lemma}

\begin{lemma} \label{l:disc}
 Let $\Sigma\subset PSL(2, \C)$ be a non discrete group, then:
\begin{enumerate}
 \item The set  $\P^1\setminus Eq(\Sigma)$ is either, empty,  one points, two points, a circle or $\P^1$.
\item  If $\mathcal{C}$ is an invariant closet set which contains at least 2 points, then $\P^1\setminus Eq(\Sigma)\subset\overline{\Sigma \mathcal{C}}$.
\item  The set $\P^1\setminus Eq(\Sigma)$ is the closure of the loxodromic fixed points.
 \end{enumerate}
\end{lemma}
\subsection{Counting Lines}
\begin{definition}
Let $\Omega\subset \P^2$  be a non-empty open set. Let us define:
\begin{enumerate}
\item The lines in general position outside $\Omega$ as:
\begin{small}
\[
LG(\Omega)=\left \{\mathcal{L} \subset Gr_1(\P^2)\vert \textrm{The lines in } \mathcal{L} \textrm{ are in general position } \&\,  \bigcup\mathcal{L}\subset  \P^2\setminus \Omega \right \}; 
\]
\end{small}
\item The number of lines in general position outside $\Omega$ as:
\[
LiG(\Omega)=max(\{card(\mathcal{L}): \mathcal{L}\in  LG(\Omega)\}), 
\]
where $card(C)$ denotes the number of elements contained in $C$.
\item Given $\mathcal{L}\in 
LG(\Omega)$ and $v\in \bigcup \mathcal{L}$, we will say that  $v$ is a vertex for  $\mathcal{L}$ if there are $\ell_1,\ell_2\in \mathcal{L}$ distinct lines and   an infinite set $\mathcal{C}\subset  Gr_1(\P^2)$  such that $\ell_1\cap\ell_2 \cap(\bigcap \mathcal{C})=\{v\}$ and 
$\bigcup\mathcal{C}\subset \P^2\setminus \Omega$.
\end{enumerate}   
\end{definition}

\begin{proposition} Let $\Gamma\subset    PSL(3\Bbb{C})$ be a complex Kleinian group. If $LiG(\Omega(\Gamma))=4$, then for each $\mathcal{L}\in GL(\Omega)$ with $card(\mathcal{L})=4$, it falls out that:
\begin{enumerate}
\item \label{4lin5} The array of lines 
$ \mathcal{L}$ contains exactly two vertexes;
\item \label{4lin2} For every  vertex $v$ of $\mathcal{L})$, it follows that $Isot(v,\Gamma)$ is a subgroup of $\Gamma$ with finite index.
\end{enumerate}
 
\end{proposition}

\section{Toral Groups}  \label{s:toral}

\begin{definition}
 Let $A\in SL(2,\Bbb{Z})$, then $A$ is said to be a Hyperbolic Toral Automorphism if none of the eigenvalues of $A$ lies on the unit circle.
\end{definition}

\begin{theorem}
 Let $A\in SL(2,\Bbb{Z})$ be an hyperbolic toral automorphism  then:

\begin{enumerate}
 \item The eigenvalues of $A$ are irrational numbers.
\item It holds  $$\{x\in \Bbb{R}^2: A^n(x)-x\in \Bbb{Z}\times \Bbb{Z} \textrm{ for some } n\in \Bbb{N}\}=\Bbb{Q}\times \Bbb{Q}.$$
\end{enumerate}
\end{theorem}

\begin{definition}
Set  $S:\Bbb{Q}\times \Bbb{Q}\rightarrow \Bbb{N}$ 
which is given by: 
\[
S(x)=min\{n\in \Bbb{N}:nx\in \Bbb{Z}\times\Bbb{Z} \}.
\]
Also define   $Per:SL(2,\Bbb{Z})\times \Bbb{Q}\times \Bbb{Q}\rightarrow \Bbb{N}$
by
\[
Per(A,x)=min\{n\in \Bbb{N}:B^n(x)-x\in \Bbb{Z}\times\Bbb{Z} \}.
\]
Finally define $\phi:SL(2,\Bbb{Z})\times\Bbb{Q}\times \Bbb{Q}\times \Bbb{Z}\rightarrow \Bbb{Q}\times\Bbb{Q}$ by
\[
\phi(B,x,l)=
\left \{
\begin{array}{ll}
 \sum_{j=0}^{l-1}B^j(x) & \textrm{ if } l>0;\\
0 & \textrm{ if } l=0;\\
-\sum_{j=1}^{l}B^{-j}(x) & \textrm{ if } l<0.\\
\end{array}
\right. .
\]

\end{definition}

The following straightforward lemmas    will be usefull 

\begin{lemma}\label{l:tec1}
Let $B\in SL(2, \Bbb{Z}) $, $\nu\in \Bbb{Q}^2$ and    $r,s,l\in \Bbb{N}$ with $0<r,s<Per(B, \nu)$ . If $l=KPer(B,\nu)+r$ and $K=\tilde nS(\nu)+t$, where $K,r, \tilde n, t\in \Bbb{N}$ are given by the division theorem,   
\[
\begin{array}{l}
\delta_1=-B^r(\phi(B, B^{-KPer(B,\nu)}(\nu)-\nu),l);\\
\delta_2=-B^{Per(B,\nu)-r}(\phi(B,B^{-Per(B,\nu)}(\nu)-\nu,l));\\
\delta_3=S(\nu)B^{Per(B,\nu)-R}(\phi(B,\nu,l));\\
\delta_4=\sum_{i=0}^{K-1}\phi(B,B^{iPer(B, \nu)}(\nu)-\nu, Per(B,\nu));\\
\delta_5= \phi(B, B^{KPer(B,\nu)}(\nu)-\nu, r);\\
\delta_6=\tilde nS(\nu)\phi(B,\nu,Per(B,\nu));\\
\delta_7=\left
\{
\begin{array}{ll}
0 & \textrm{ if } r+s<Per(B,\nu)\\
\phi(B,B^{Per(B,\nu)}(\nu)-\nu, r+s-Per(B,\nu) ) & \textrm{D.O.F.}
\end{array}
\right.\\
\delta_8=\left
\{
\begin{array}{ll}
\sum_{j_r}^{r+s}B^j(\nu) & \textrm{ if } r+s<Per(B,\nu)\\
\sum_{j_r}^{Per(B,\nu)-1}B^j(\nu)+ \phi(B,\nu, r+s-Per(B,\nu) ) & \textrm{D.O.F.}
\end{array}
\right. ,
\end{array}
\]
 thus $\delta_1,\delta_2,\delta_3,\delta_4,\delta_5, \delta_6,\delta_7\in \Bbb{Z}^2$ and  
\begin{equation}
\begin{array}{l}
 \phi(B,\nu,-l )=\delta_1+\delta_2+\delta_3+(S(\nu)-1)B^{Per(B,\nu)-r}(\phi(B,\nu,l));\\ 
 \phi(B,\nu,l )=\delta_4+\delta_5+\delta_6+t\phi(B,\nu,Per(B,\nu))+\phi(B,\nu,r); \\
B^r\phi(B,\nu, s+1)=\delta_7+\delta_8
\end{array}
\end{equation}
\end{lemma}

\begin{lemma} \label{l:tec2}
 Let $(a_m), (b_m)\subset \Bbb{C}$ be sequences, then:
\begin{enumerate}
 \item If $(a_m)$ and $ (b_m)$ diverges, then the accumulation points of $$\{[a_m:b_m:1]:m\in \Bbb{N}\}$$ lies on $\overleftrightarrow{e_1,e_2}$;
\item If $(a_m)$ converges and $ (b_m)$ diverges, then $[a_m:b_m:1]\xymatrix{ \ar[r]_{m \rightarrow  \infty}&} [e_2]$;
\item  If $(a_m)$ diverges and $ (b_m)$ converges, then $[a_m:b_m:1]\xymatrix{ \ar[r]_{m \rightarrow  \infty}&} [e_1]$;
\item 
If  $k_m=[a_m:b_m:1]\xymatrix{ \ar[r]_{m \rightarrow  \infty}&} [z:0:1],$ where $z\neq 0$, then there is a subsequence of $(k_m)$, denoted $(\tilde k_m=[\tilde a_m:\tilde b_m:1])$, such that $(\tilde a_m)$ and  $(\tilde b_m)$ are convergent.
\end{enumerate}
\end{lemma}

\begin{definition} \label{d:gr}
Let  $A,B\in SL(2,\Bbb{Z})$,  $\nu\in \Bbb{ Q}\times\Bbb{ Q}$,  $b\in M(1\times 2,\Bbb{Z})$, $k,l\in \Bbb{Z}$, thus we define: 

\[
\left \langle k:l:b:\nu\right \rangle=
\left (
\begin{array}{lll}
A^kB^l  & b+\phi(\nu,l) \\
 0 & 1\\
\end{array}
\right ).
\]
\end{definition}

From Lemma \ref{l:tec1}, it follows easily

\begin{corollary} \label{c:for}
 Let  $A,B\in SL(2,\Bbb{Z})$,  $\nu\in \Bbb{ Q}\times\Bbb{ Q}$,  $b\in M(1\times 2,\Bbb{Z})$, $k,l\in \Bbb{Z}$, then there are $m_0, \ldots, m_{Per(B,\nu)-1}\in \{0, \ldots , S(\nu)-1\}$ and $B\in \Bbb{Z}^2$ such that:
\[
 \left \langle k:l:b:\nu\right \rangle=
\left (
\begin{array}{lll}
A^kB^l  & B+\sum_{j=0}^{Per(B,\nu)-1}m_j B^j(\nu) \\
 0 & 1\\
\end{array}
\right ),
\]
  
\end{corollary}

\begin{proposition} \label{p:tor2}
Let  $A,B\in SL(2,\Bbb{Z})$  be such that  the group generated by $A,\,B$ is isomorphic to $\Bbb{Z}\times  \Bbb{Z}$ and  each element in $<A,B>\setminus\{Id\}$ is a hyperbolic toral automorphism, also let  $\nu\in \Bbb{ Q}\times\Bbb{ Q}$ be such that   $A(\nu)-\nu\in \Bbb{Z}\times  \Bbb{Z}$. 
Then  
\[
 \Gamma_{A,B,\nu}=\left \{\left \langle k:l:b:\nu\right \rangle\mid  k,l\in \Bbb{Z}, \,b\in M(1\times 2,\Bbb{Z} )\right   \}
\]
is a complex Kleinian  group.
Moreover $\Omega(\Gamma_{A,B,\nu})$ is projectively equivalent to: $$\bigcup_{\epsilon_1,\epsilon_2\in\{\pm 1\}} \Bbb{H}^{\epsilon_1}\times \Bbb{H}^{\epsilon_2}.$$
\end{proposition}
\begin{proof}
Let $a=\langle k_1:l_1:b_1:\nu \rangle,b=\langle k_2:l_2:b_2:\nu\rangle\in \Gamma_{A,B,\nu}$. Thus an easy calculation shows:
\[
ab^{-1}=\langle k_1-k_2:l_1-l_2:b_1+b_3+b_4\rangle
\]
where 
\[
\begin{array}{l}
b_3=-A^{k_1}B^{l_1} (A^{-k_2}B^{-l_2}(b_2)+B^{-l_2}\phi(A^{-k_2}(\nu)-\nu,l_2));\\
 b_4=B^{l_1}(\phi(A^{k_1}(\nu)-\nu, -l_2)).
\end{array}
\]

Since $b_3,b_4\in \Bbb{Z}\times \Bbb{Z}$, it follows that 
 $\Gamma_{A,B,\nu}$ is a group.

Now, since $A,B\in SL(2,\Bbb{Z})$ are commuting  hyperbolic toral automorphism, it follows  that there is $\hat T\in SL(2,\Bbb{R})$ such that $\hat TA\hat T^{-1},\hat TB\hat T^{-1}$ are diagonal matrices. Set 
\[
T=
 \left (
\begin{array}{ll}
\hat T & 0 \\
 0 & 1\\
\end{array}
\right);\,
TAT^{-1}=
 \left (
\begin{array}{ll}
\alpha & 0 \\
 0 & \alpha^{-1}\\
\end{array}
\right);\,
TBT^{-1}=
 \left (
\begin{array}{ll}
\beta & 0 \\
 0 & \beta^{-1}\\
\end{array}
\right);\,
\] 
where $\alpha,\beta\in \Bbb{R}\setminus\{\pm 1\}$ and   $\hat T(1,0),\hat T(0,1)\in \Bbb{R}^2$. Moreover, given   $b\in M(1\times 2,\Bbb{Z})$ and $k,l\in \Bbb{Z}$ and taking $\nu=(\nu_1,\nu_2)$, by Corollary \ref{c:for} there are $m_0, \ldots, m_{Per(B,\nu)-1}\in \{0, \ldots , S(\nu)-1\}$ and $b_1,b_2\in \Bbb{Z}$ such that
\begin{tiny}
\begin{equation}\label{e:fortor}
T 
\left \langle k:l:b:\nu\right \rangle
T^{-1}=
 \left (
\begin{array}{lll}
\alpha^k\beta^m & 0                     & \sum_{i=1}^2x_i(b_i+\nu_i\sum_{j=0}^{Per_B(\nu)-1}m_j\beta ^j) \\
 0              & \alpha^{-k}\beta^{-m} & \sum_{i=1}^2y_i(b_i+\nu_i\sum_{j=0}^{Per_B(\nu)-1}m_j\beta ^{-j}) \\
 0              & 0                     &1\\
\end{array}
\right).
\end{equation}
\end{tiny}
Claim 1. Let $z\neq 0$, if $[z:0:1]$ lies on  $L_1(T\Gamma_{A,B,\nu}T^{-1})$, then $z\in \Bbb{R}$.  Let us assume that $Im(z\neq 0)$.  Thus there are $w=[a:b:1]$ and $(\gamma_m)\subset \Gamma_{A,B,\nu}$ a sequence of distinc elements in $ \Gamma_{A,B,\nu}^{-1}$ such that $T\gamma_mT^{-1}(w) \xymatrix{ \ar[r]_{m \rightarrow  \infty}&} x$. From equation (\ref{e:fortor}),  it follows that  for each $m\in \Bbb{N}$ there are $n_m,k_m,b_{1m}, b_{2m}\in \Bbb{Z}$ and $\{l_{j}\}_{j=0}^{Per_B(\nu)-1}\in \{0,\ldots, S(\nu)-1\}$  such that
$\gamma_m(w)=
[a_m:b_m:1]$
where 
\[
a_m=\alpha^{k_m}\beta^{n_m}a + \sum_{i=1}^2x_i\left (b_{im}+\nu_i\sum_{j=0}^{Per_B(\nu)-1}l_{j}\beta ^j\right ) ;        \]
\[
b_m=\alpha^{-k_m}\beta^{-n_m}b+ \sum_{i=1}^2y_i\left (b_{im}+\nu_i\sum_{j=0}^{Per_B(\nu)-1}l_{j}\beta ^{-j}\right ).
\]
By Lemma  \ref{l:tec2}, we can assume that $a_m\xymatrix{ \ar[r]_{m \rightarrow  \infty}&} z$ and $b_m \xymatrix{ \ar[r]_{m \rightarrow  \infty}&} 0$.
Since $Im(a_m)=\alpha^{k_m}\beta^{n_m}Im(a) \rightarrow Im(z)\neq 0$. We conclude that $(k_m)$ and $(n_m)$ are eventually constant.  In consequence we conclude that $(p_{1m})$ and $(p_{2m})$ are eventually constant. Thus $(\gamma_m)$ is eventually constant. Which is contradiction. \\

Observe that by a similar argument, we can show the  claim in the case  $x=[0:z:1]\in L_1(T\Gamma_{A,B,\nu}T^{-1})$.\\

Now let $\gamma\in T\Gamma_{A,B,\nu}T^{-1}$ induced by the linear map:
\begin{equation}
\left (
\begin{array}{lll}
\alpha & 0           & 0 \\
 0     & \alpha^{-1} & 0\\
 0     & 0           &1\\
\end{array}
\right),
 \end{equation}
then is straighfoward to check   that  $\overleftrightarrow{e_1,e_3}\cup \overleftrightarrow{e_2,e_3}\subset \Lambda(T\Gamma_{A,B,\nu}T^{-1})$.

On the other hand, from equation (\ref{e:fortor}), we conclude that $\ell_{1}=\langle\{e_2,e_3\} \rangle$, $\ell_{2}=\langle\{e_1,e_3\} \rangle$, $e_1$ and   $e_2,$ are $T\Gamma_{A,B,\nu}T^{-1}$-invariant. Thus,  taking    ,   $\pi_i=\pi_{e_i,\ell_i}$ we can define  $\Pi_i:\Gamma_0\rightarrow Bihol(l_i)$. Thus, from Equation  \ref{e:fortor}, we conclude that  $\Pi_j(T\Gamma_{A,B,\nu}T^{-1})$ leaves  $e_k$, $k\in\{1,2\}\setminus\{i\}$, and $$[\{r\alpha_1e_k+se_3\vert r,s\in \Bbb{R}\}\setminus\{0\}]$$ invariant, moreover, it   contains loxodromic and parabolic elements. Thus Lemma \ref{l:disc},  yields $$\ell_j\setminus Eq(\Pi_j(T\Gamma_{A,B,\nu}T^{-1}))=[\{r\alpha_1e_k+se_3\vert r,s\in \Bbb{R}\}\setminus\{0\}].$$  Thus a straightforward calculation shows 
\[
\overline{T\Gamma_{A,B,\nu}T^{-1}(\ell_1\cup \ell_2)}= \P^2\setminus \bigcup_{j\in \{1,2\}}\bigcup_{p\in \Bbb{R}(\ell_j)}\overleftrightarrow{e_j,p}\subset \Lambda(T\Gamma_{A,B,\nu}T^{-1}).
\]
Thus $\Omega=\bigcup_{\epsilon_1,\epsilon_2\in\{\pm 1\}} \Bbb{H}^{\epsilon_1}\times \Bbb{H}^{\epsilon_2}$ is an open $T\Gamma_{A,B,\nu}T^{-1}$-invariant set, with $LiG(\Omega)=4$. In consequence Theorem 3.5 in \cite{bcn},  yields $\Omega\subset Eq(T\Gamma_{A,B,\nu}T^{-1}))\subset  \Omega(\Gamma)$. Which clearly concludes the proof. 
\end{proof}

By means of similar arguments the following proposition can be showed.

\begin{proposition} \label{p:tor1}
 Let  $A\in SL(2,\Bbb{Z})$  be an hyperbolic toral automorphism, then the following set is a discrete  group of $PSL(3,\C)$

\[
\Gamma_A=
\left \{
\left (
\begin{array}{lll}
A^k  & b \\
 0 & 1\\
\end{array}
\right )\vert b\in M(1\times 2,\Bbb{Z}), \,k\in \Bbb{Z} k\in \Bbb{Z},  
\right \}
\]
Moreover $\Omega(\Gamma_{A,B,\nu})$ is projectively equivalent to $\bigcup_{\epsilon_1,\epsilon_2\in\{\pm 1\}} \Bbb{H}^{\epsilon_1}\times \Bbb{H}^{\epsilon_2}$.
\end{proposition}

\begin{definition} A subgroup $\Gamma\subset PSL(3,\C)$ is said to be a {\it Hyperbolic Toral Group} if $\Gamma$ is conjugated to the group described either in proposition \ref{p:tor1} or the one in proposition \ref{p:tor2}.
 
\end{definition}

\section{Four lines Groups} \label{s:4lines}
Trough this section $\Gamma\subset PSL(3,\Bbb{C})$ is  a complex Kleinian group with  $LiG(\Omega(\Gamma))=4$, $\mathcal{L}\in LG(\Omega(\Gamma))$ with $card(\mathcal{L})=4$, the vertex of $ \mathcal{L}$ are $e_1,e_2$,  $\ell_{1}=\overleftrightarrow{e_2,e_3} $, $\ell_{2}=\overleftrightarrow{e_1,e_3}$, $\Gamma_0=Stab(e_1,\Gamma)\cap Stab(e_2,\Gamma)$,   $\pi_i=\pi_{e_i,\ell_i}$ and $\Pi_i=\Pi_{e_i,\ell_i}$.\\

\begin{lemma} \label{l:exlox}
 Either $\Pi_1(\Gamma_0)$  or $\Pi_2(\Gamma_0)$ contains  loxodromic elements. 
\end{lemma}

\begin{proof}
 On the contrary, let us assume that $\Pi_j(\Gamma_0)$, $j\in\{1,2\}$, does not contains loxodromic elements. Thus each element $\gamma\in \Gamma$ has a lift $\tilde \gamma\in GL(3,\C)$ which is given by:
\[
 \tilde \gamma=
\left (
\begin{array}{lll}
\gamma_{11} &  0          & \gamma_{13}\\
0           & \gamma_{22} & \gamma_{23}\\
0            & 0            & 1\\
\end{array}
\right )
\]
where $\vert \gamma_{11}\vert =\vert\gamma_{22}\vert =1$.  A straightforward calculation shows that $Eq(\Gamma_0)=\P^2\setminus \overleftrightarrow{e_1,e_2}$. Thus   $Eq(\Gamma)=Eq(\Gamma_0)$. Which is a contradition. 
\end{proof}

\begin{lemma}
 If $\Pi_{i_0}(\Gamma_0)$ contains  a loxodromic element then $\bigcap_{\tau \in \Pi_{i_0}(\Gamma_0)}Fix(\tau)$ contains a single point.
\end{lemma}

\begin{proof}
 Without loss of generality we may assume that $i_0=2$. Now, if $F_2=\bigcap_{\tau \in \Pi_2(\Gamma_0)}Fix(\tau)$ contains more than one point, we deduce that $F_2=\{e_1,z\}$ for some  $z\in \ell_2\setminus \{e_1\}$. By conjugating by a projective transformation, if it is necessary, we may assume that $z=e_3$. Thus each element $\gamma\in \Gamma_0$ has a lift $\tilde \gamma\in SL(3,\C)$ which is given by:
\[
 \tilde \gamma=
\left (
\begin{array}{lll}
\gamma_{11} &  0          & 0\\
0           & \gamma_{22} & \gamma_{23}\\
0            & 0            & 1\\
\end{array}
\right )
\]
where $abc=1$. In consequence $\ell_1$ and $e_1$ are invariant under the action of $\Gamma_0$.  By Lemma \ref{l:control}, $W=\P^2\setminus(\ell_1\cup \{e_1\})$ is a discontinuity region for $\Gamma_0$ which is contained in $Eq (\Gamma_0)$. In consequence $Lin(\Gamma W)<\infty$, which is a contradiction, since $\Gamma W\subset \Gamma(Eq(\Gamma_0)=Eq(\Gamma)$.
\end{proof}

\begin{lemma}
  The groups  $\Pi_{1}(\Gamma_0)$ and  $\Pi_{2}(\Gamma_0)$ contains   loxodromic elements.
\end{lemma}

\begin{proof} By Lemma \ref{l:exlox}, either $\Pi_1(\Gamma_0)$ or $\Pi_2(\Gamma_0)$ contains loxodromic elements.
 Without loss of generality let us assume that $\Pi_1(\Gamma_0)$ contains a loxodromic element. If $\Pi_2(\Gamma_0)$ does not contains loxodromic elements, every element $\gamma\in \Gamma_0$ has a lift $\gamma\in GL(3,\C)$ which is given by:

\[
 \tilde \gamma=
\left (
\begin{array}{lll}
\gamma_{11}   &  0          & \gamma_{13}\\
0             & \gamma_{22} & \gamma_{23}\\
0              & 0            & 1\\
\end{array}
\right )
\]
where $\vert \gamma\vert =1$. In consequence there are $\gamma,\tau\in \Gamma_0$  such that $\Pi_1(\gamma)$ and $\Pi_1(\tau )$ are loxodromic elements with $Fix (\Pi_1(\gamma)) \neq Fix(\Pi_1(\tau ))$,  $\Pi_2(\gamma)$ is  either parabolic or the identity and $\Pi_2(\tau)$ is  either parabolic or the identity. An easy calculation shows $\Pi_1(\tau\gamma\tau^{-1}\gamma^{-1})$ is parabolic and $\Pi_1(\tau\gamma\tau^{-1}\gamma^{-1})$ is the identity. In consequence $\kappa=\tau\gamma\tau^{-1}\gamma^{-1}$ has a lift $\tilde \kappa\in SL(3,\C)$ given by:
 
\[
 \tilde \gamma=
\left (
\begin{array}{lll}
1  &  0 & 0\\
0  & 1  & \kappa_{23}\\
0  & 0  & 1\\
\end{array}
\right ).
\]
 Finally, let  $\gamma_0\in \Gamma_0$ be such that $\Pi_1(\gamma_0)$ is a loxodromic  element. By conjugating with a projective transformation, if it is necessary, we may assume that  $Fix(\Pi_1(\gamma_0))=\{e_2,e_3\}$. Also, by taking the Inverse of $\gamma$, if it is necessary, we may assume that $e_2$ is an attracting point for $\Pi_1(\gamma_0)$.  In consequence, if $\tilde \gamma_0=(\gamma_{ij}\in SL(3,\C))$ is a lift of $\gamma$, we conclude that: 

\[
 \tilde \gamma^m\tilde \kappa\tilde\gamma^{-m}=
\left (
\begin{array}{lll}
1  &  0 & 0\\
0  & 1  &\gamma_{22} ^m\kappa_{23}\\
0  & 0  & 1\\
\end{array}
\right ).
\xymatrix{ \ar[r]_{m \rightarrow  \infty}&}
Id,
\]
which is a contradiction, since $\Gamma$ is discrete.
\end{proof}

\begin{lemma}
There is an element $\gamma_0\in \Gamma_0$ such that    $\Pi_{1}(\gamma_0)$ and  $\Pi_{2}(\gamma_0)$ are loxodromic.
\end{lemma}

\begin{proof}
If this is not the case,  there are $\gamma_1,\gamma_2\in \Gamma_0$ such that $\Pi_j(\gamma_k)$ is loxodromic if $j= k$ and either the identity or parabolic  in other case. Is straightforward to check that  $\Pi_j(\gamma_1 \gamma_2)$, $j\in \{1,2\}$, is loxodromic. Which is a contradiction. 
\end{proof}

From now on  $\gamma_L$ will denote a fixed element in $\Gamma_0$ such that $\Pi_j(\gamma_L)$, $j\in\{1,2\}$, is loxodromic. Also, by conjugating with a projective transformation, if it is necessary, we may assume that $\gamma_L$ has a lift $\tilde{\gamma}_L=
(\gamma_{Lij})$ which is a diagonal matrix.

\begin{lemma} \label{l:loxlox}
There is an element $\gamma\in \Gamma_0$ such that    $\Pi_{1}(\tau_0)$ and  $\Pi_{2}(\tau_0)$ are  parabolic elements.
\end{lemma}

\begin{proof}
 Let $j\in \{1,2\}$, then   there is an element $\gamma_j$ such that $\Pi_{j}(\gamma_j)$  is loxodromic and  $Fix(\Pi_{j}(\gamma_j))\neq Fix(\Pi_{j}(\gamma_L)) $. Set $\kappa_j=\gamma_L\gamma_j\gamma_L^{-1}\gamma_j^{-1}$, then $\Pi_i(\kappa_j)$ is parabolic if $i=j$ and is either the identity or parabolic in other case. Thus the only interesting case is $\Pi_1(\kappa_2)=Id$ and $\Pi_2(\kappa_1)=Id$. But in such case a simple calculation shows that $\Pi_j(\kappa_1\kappa_2)$, $j\in\{1,2\}$, is parabolic.  
\end{proof}

From now on  $\gamma_P$ will denote a fixed element in $\Gamma_0$ with a lift $(\gamma_{Pij}),$  such that $\Pi_j(\gamma_0)$, $j\in\{1,2\}$, is parabolic.

\begin{lemma}
If $\mid\gamma_{L11}\mid <\mid\gamma_{L33}\mid $, then $\mid\gamma_{L22}\mid >\mid\gamma_{L33}\mid $.
\end{lemma}

\begin{proof}
 On the contrary, let us assume that $\vert \gamma_{L22}\vert<\vert \gamma_{L33}\vert $. Then a straightforward calculations shows that:
\[
\kappa_m=\left (
\begin{array}{lll}
1  & 0 & \gamma_{L11}^m\gamma_{P13}\gamma_{L33}^{-m}\\
0  & 1 & \gamma_{L22}^m\gamma_{P23}\gamma_{L33}^{-m}\\
0  & 0 & 1\\
\end{array}
\right ).
\]
is a lift of $\gamma_L^{-m}\gamma_P\gamma_L^m$. And clearly $[[\kappa_m]] \xymatrix{ \ar[r]_{m \rightarrow  \infty}&}  Id$, which is a contradiction, since $\Gamma_0$ is discrete.
\end{proof}

\begin{lemma}
The sets    $\P^2\setminus Eq(\Pi_{1}(\Gamma_0))$ and  $\P^2\setminus Eq(\Pi_{2}(\Gamma_0))$ are circles.
\end{lemma}
\begin{proof}
 Since the vertex of $\mathcal{L}$ are $ e_1,e_2$ it follows that 
\[
\mathcal{C}_j=\Pi_j\left (\bigcup \{\ell\in Gr_1(\P^2)\vert e_j \in \ell, \ell\subset \Lambda(\Gamma)\}\right ),
\]
is a closet, infinite and  $\Pi_j(\Gamma_0)$-invariant set. Thus  Lemma \ref{l:disc} yields the result.
 
\end{proof}

\begin{lemma}
Up to conjugacy  $\Gamma_0$ leaves $\Bbb{P}^3_\Bbb{R}$ invariant. 
\end{lemma}

\begin{proof}
By Lemma \ref{l:loxlox} there is an element  $\gamma_0\in \Gamma_0$ such that  $\Pi_1(\gamma_0)$ and $\Pi_2(\gamma_0)$ are loxodromic elements. Thus after conjugating with a projective transformation, if it is necessary, we may assume that $Fix(\Pi_1(\gamma_0))=\{[e_2],[e_3]\}$ and that $Fix(\Pi_2(\gamma_0))=\{[e_1],[e_3]\}$. In consequence $\gamma_0$ has a lift $\tilde\gamma_0\in SL(3,\C)$ given by:
\[
 \tilde \gamma_0=
\left (
\begin{array}{lll}
\gamma_{11}  &  0 & 0\\
0  & \gamma_{22} & 0\\
0  & 0  & \gamma_{22}\\
\end{array}
\right ),
\]
where $\gamma_{11}\gamma_{22}\gamma_{33}=1$. Thus there are $\alpha_1,\alpha_2\in \C^*$ such that:
\[
 \ell_1\setminus Eq(\Pi_1(\Gamma_{0})=[\{r\alpha_1e_2+se_3\vert r,s\in \Bbb{R}\}\setminus\{0\}];
\]
\[
 \ell_2\setminus Eq(\Pi_2(\Gamma_{0})=[\{r\alpha_2e_1+se_3\vert r,s\in \Bbb{R}\}\setminus\{0\}].
\]
Let $\eta\in PSL(3,\C)$ be the element induced by the linear map:
\[
\tilde \eta=
\left (
\begin{array}{lll}
\alpha_2  &  0 & 0\\
0  & \alpha_1 & 0\\
0  & 0  & 1\\
\end{array}
\right ).
\]
Thus a straightforward calculation shows that 
\[
\Pi_1(\eta^{-1}\Gamma_0\eta)[\{re_2+se_3\vert r,s\in \Bbb{R}\}\setminus\{0\}]=[\{re_2+se_3\vert r,s\in \Bbb{R}\}\setminus\{0\}];
\]
\[
\Pi_2(\eta^{-1}\Gamma_0\eta)[\{re_1+se_3\vert r,s\in \Bbb{R}\}\setminus\{0\}]=[\{re_1+se_3\vert r,s\in \Bbb{R}\}\setminus\{0\}].
\]
In consequence $\Bbb{P}^2_\Bbb{R}$ is $\eta^{-1}\Gamma_0\eta$-invariant.
\end{proof}

From now on we will assume that $\Bbb{P}^2_\Bbb{R}$ is $\Gamma_0$-invariant. Also let us define:

\[
Par(\Gamma_0)=\{\gamma\in \Gamma_0\vert \Pi_j(\gamma),j\in \{1,2\}, \textrm{ is either parabolic or the identity} \}. 
\]

\begin{lemma}
  The set  $Par(\Gamma_0)$ is a group isomorphic  to $\Bbb{Z}\times \Bbb{Z}$.
\end{lemma}

\begin{proof}
 Clearly $ Par(\Gamma_0)$ is a group. Moreover $ Par(\Gamma_0)$ can be lifted to a group $\widetilde{Par(\Gamma_0)}\subset SL(3,C)$ where each element has the form:
\[
\left (
\begin{array}{lll}
1  & 0 & a\\
0  & 1 & b\\
0  & 0 & 1\\
\end{array}
\right ).
\]
where $a,b\in \Bbb{R}$. Also observe that  the group morphism $Lat:\widetilde{Par(\Gamma_0)}\rightarrow \Bbb{R}^2$ given by $Lat ((\gamma_{ij}))=(\gamma_{13},\gamma_{23})$ enable us to show that $ Par(\Gamma_0)$ is isomorphic to a lattice in $\Bbb{R}^2$. Thus to get the claim, will be enought to show that there  are two elements in $Lat(\widetilde{Par(\Gamma_0)})$ which are $\Bbb{R}$-linearly independent. Clearly 
\[
\kappa_1=\left (
\begin{array}{lll}
1  & 0 & \gamma_{L11}\gamma_{P13}\gamma_{L33}^{-1}\\
0  & 1 & \gamma_{L22}\gamma_{P23}\gamma_{L33}^{-1}\\
0  & 0 & 1\\
\end{array}
\right ).
\]
is a lift in $\widetilde{Par(\Gamma_0)}$ of $\gamma_{L}^{-1}\gamma_{P}\gamma{L}$. To conclude  observe that the system of linear equations 
\[
rLa(\kappa_1)+sLa(\gamma_P)=0
\] 
has determinant $\gamma_{P23}\gamma_{P13}\gamma_{L33}^{-1}(\gamma_{L11}-\gamma_{L22})\neq 0$. Which conclude the proof.
\end{proof}

Set
\[
\Bbb{R}(\ell_1)=[\{r\alpha_1e_2+se_3\vert r,s\in \Bbb{R}\}\setminus\{0\}];
\]
\[
\Bbb{R}(\ell_2)=[\{r\alpha_1e_1+se_3\vert r,s\in \Bbb{R}\}\setminus\{0\}];
\]

\begin{proposition}
 The equicontinuity set of $\Gamma$ is given by:
\[
Eq(\Gamma)=\bigcup_{\epsilon_1,\epsilon_2\in\{\pm 1\}} \Bbb{H}^{\epsilon_1}\times \Bbb{H}^{\epsilon_2}.
\]
\end{proposition}
\begin{proof}
 Let us set
\[
\Omega =\P^2\setminus \bigcup_{j\in \{1,2\}}\bigcup_{p\in \Bbb{R}(\ell_j)}\overleftrightarrow{e_j,p}.
\]
Clearly $\Omega=\bigcup_{\epsilon_1,\epsilon_2\in\{\pm 1\}} \Bbb{H}^{\epsilon_1}\times \Bbb{H}^{\epsilon_2}$ and is a $\Gamma_0$-invariant with $LiG(\Omega)=4$. Thus Theorem 3.5 in \cite{bcn}, yields $\Omega\subset Eq(\Gamma_0).$ On the other hand, let $j\in\{1,2\}$  and $l_j\in Gr_1(\P^2)$ be such that $e_j\in l_j\subset \Lambda(\Gamma)$. Then 
\[
\bigcup_{p\in \Bbb{R}(\ell_j)}\overleftrightarrow{p,e_j}\subset \bigcup_{\gamma\in(\Gamma_0)}\overleftrightarrow{e_j,\Pi_j(\gamma)(\pi_j(l_j\setminus\{e_j\})) }= \overline{\Gamma l_j}.
\]
In consequence $Eq(\Gamma)\subset \Omega$. Finally, since $\Gamma_0$ is a subgroup of finite index of $\Gamma$, we conclude that $Eq(\Gamma_0)=Eq(\Gamma) $, which concludes the proof.
\end{proof}

In the sequel  let $\check \Gamma=\bigcap_{\epsilon_1,\epsilon_2\in \{\pm 1\}} Stab(\Bbb{H}^{\epsilon_1}\times \Bbb{H}^{\epsilon_2}, \Gamma)$. As an immediate consequence one has:

\begin{lemma}
 The group $\check \Gamma$ is a subgroup of $\Gamma$ of index at most 4, which contains $Par(\Gamma_0)$. Moreover $\check \Gamma$ can be lifted to a subgroup $\tilde \Gamma$ of $GL(3,\Bbb{R})$ where each element has the form:
\[
\left (
\begin{array}{lll}
a & 0 & c\\
0 & b & d\\
0 & 0 & 1
\end{array}
\right )
\]
where $a,b>0$ and $c,d\in \Bbb{R}$.

\end{lemma}
\begin{definition}
From now on $\widetilde{Par(\Gamma_0)}$ will denote  the lift of $Par(\Gamma_0)$ in $\tilde \Gamma$ and $Lat:\widetilde{Par(\Gamma_0)}\rightarrow \Bbb{R}^2 $ will denote the group morphism given by $Lat(\gamma_{ij}) =(\gamma_{13},\gamma_{23})$.
 
\end{definition}

\begin{corollary}
 For each $j\in \{1,2\}$, it follows that  $\Pi_j(\check \Gamma)$ does not contains elliptic elements.
\end{corollary}

\begin{lemma} \label{l:par2}
 Let $j\in \{1,2\}$ and $\gamma\in \check \Gamma$. If $\Pi_j(\gamma)$ is parabolic, then $\Pi_i(\gamma)\neq Id$, where $j$ is the unique element in $\{1,2\}\setminus\{j\}$.
\end{lemma}
\begin{proof}
 On the contrary let us assume that there $j\in \{1,2\}$ and $\gamma\in \check \Gamma$ such that  $\Pi_j(\gamma)$ is parabolic and $\Pi_i(\gamma)= Id$. Now by taking the inverse of $\gamma_L$ if it is necessary, we may assume that $\mid \gamma_{L13}\mid <1$. Thus an easy calculation shows:
\[
\gamma_L^m\gamma \gamma^{-m}=
\left (
\begin{array}{lll}
1 & 0 & \gamma_{L13}^m\gamma_{13} \\
0 & 1 & 0\\
0 & 0 & 1\\
\end{array}
\right )
\xymatrix{ \ar[r]_{m \rightarrow  \infty}&} Id.
\]
Which is a contradiction.
\end{proof}

\begin{lemma} \label{l:lox2}
 Let $j\in \{1,2\}$ and $\gamma\in \check \Gamma$. If $\Pi_j(\gamma)$ is loxodromic, then $\Pi_i(\gamma)$ is loxodromic, where $j$ is the unique element in $\{1,2\}\setminus\{j\}$.
\end{lemma}
\begin{proof}
 On the contrary let us assume that there $j\in \{1,2\}$ and $\gamma\in \check \Gamma$ such that  $\Pi_j(\gamma)$ is loxodromic  and $\Pi_i(\gamma)$ is not loxodromic. Thus $\pi_i(\gamma)$ is either $Id$ or a parabolic element. Set $\tau=\gamma\gamma_P$, then $\pi_j(\tau)$ is loxodromic with $Fix(\Pi_j(\tau))\neq Fix(\Pi_j(\gamma))$. Thus an easy calculation shows that $\Pi_j(\gamma\tau\gamma^{-1}\tau^{-1})$ is parabolic  and $\Pi_i(\gamma\tau\gamma^{-1}\tau^{-1})$ is  the identity.  Which contradicts Lemma \ref{l:par2}.
\end{proof}

\begin{definition}
 Let $eLat:\tilde \Gamma \rightarrow GL(2, \Bbb{R}^+)$ given by 
\[
eLat(\gamma_{ij})=
\left (
\begin{array}{ll}
 \gamma_{11} & 0\\
0 & \gamma_{22}
\end{array}
\right )
\]
\end{definition}

\begin{proposition}
 $eLat$ is a group morphism whose  kernel is $\widetilde{Par(\Gamma_0)}$.
\end{proposition}

\begin{proposition} \label{p:tor}
$Lat(\widetilde{Par(\Gamma_0)})$ is invariant, as set of $\Bbb{R}^2$, under the action of the group $eLat(\tilde  \Gamma)$.   
\end{proposition}
\begin{proof}
 Let $(a,b)\in Lat(\widetilde{Par(\Gamma_0)})$ and $(\gamma_{ij})\in eLat(\tilde  \Gamma)$. Thus there is $\gamma\in \tilde  \Gamma$ and $\tau\in \widetilde{Par(\Gamma_0)} $ such that $eLat(\gamma)= (\gamma_{ij})$ and $Lat(\tau)=(a,b)$.  Thus $\kappa=\gamma\tau\gamma^{-1}\in \widetilde{Par(\Gamma_0)}$ and $Lat(\kappa)=(\gamma_{11}a,\gamma_{22}b) $. 
\end{proof}
\begin{lemma}
 $eLat(\tilde  \Gamma)$ is a conmutative discrete group  with at most 2 generators. 
\end{lemma}

 \begin{proof}
  Since the map  $\rho:eLat(\check \Gamma)\rightarrow \Bbb{R}^2$ given by
\[
\rho
\left (
\begin{array}{ll}
a & 0\\
0 & b
\end{array}
\right )
=(Log(a),Log(b))
\]
is an isomorphism of groups and from the Bierberbach Theorem (see \cite{rat}),  it will be enought to show  $eLat(\check\Gamma)$ is discrete. If this is not the case, then there are sequences of distinc elements $(\alpha_m)$, $(\beta_m)\in \Bbb{R}^+$ such that $\left (
\begin{array}{lll}
\alpha_m & 0\\
0 & \beta_m
\end{array}
\right )
\in eLat (\check \Gamma)
$
 and $\alpha_m,\,\beta_m\xymatrix{ \ar[r]_{m \rightarrow  \infty}&}
1$. Let $\gamma_m=(\gamma_{i,j}^{(m)})\in \tilde \Gamma$ such that $eLat(\gamma_m)=(\alpha_m,\beta_m)$. Since $Lat(\widetilde{Par(\Gamma_0)}) $ is a lattice of rank 2, it follows that there is a sequence $(\tau_m)\in Lat(\widetilde{Par(\Gamma_0)})$ such that $(Lat(\tau_m)+(\gamma_{13},\gamma_{23}))$ is a bounded sequence. Thus we can assume that there are $c,d\in \Bbb{R}$ such that  $(Lat(\tau_m)+(\gamma_{13}^{(m)},\gamma_{23}^{(m)}))\xymatrix{ \ar[r]_{m \rightarrow  \infty}&} (c,d )$. Now a straightforward calculation shows:
\[
\gamma_m\tau_m=
\left (
\begin{array}{ll}
eLat(\gamma_m) & Lat(\tau_m)+(\gamma_{13}^{(m)}, \gamma_{23}^{(m)})\\
0 & 1
\end{array}
\right  )
\xymatrix{ \ar[r]_{m \rightarrow  \infty}&}
\left (
\begin{array}{lll}
1 & 0 & c\\
0 & 1 & d\\
0 & 0 & 1
\end{array}
\right ),
\]
which is a contradiction.
 \end{proof}

In the sequel  $\sigma$ is a subset of $\check \Gamma$ such that $Card(\sigma)=rank (eLat(\tilde \Gamma))$ and its lift $\tilde \sigma$ in $\tilde \Gamma$ satisfies $<\tilde\sigma>=eLat(\tilde \gamma)$. Also, take  $\tau_0\in \sigma$ be a a fixed element, thus, by conjugating with a projective transformation if it is necessary, we may assume that $\tau_0$ has a lift $\tilde \tau\in \tilde \Gamma$ which is a diagonal matrix. Finally, let  $\gamma_,\gamma_2\in Par(\Gamma_0)$ be  such that $<\gamma_1,\gamma_2>=Par(\Gamma_0)$. 

\begin{lemma}
 It holds that $<\sigma,\gamma_1,\gamma_2>=\check \Gamma$.
\end{lemma}
\begin{proof}
Let $\gamma\in \check \Gamma$ and  $\tilde \gamma\in \tilde \Gamma$ be a lift. Thus there is  $\tau \in <\sigma>$ with a lift $\tilde \gamma\in \tilde \Gamma$ such that $eLat(\tilde \gamma)=eLat(\tilde \tau)$. In consequence $\gamma\tau^{-1}\in Par(\Gamma_0)$, which concludes the proof.  
\end{proof}

\begin{lemma}
 It is verified that $\check \Gamma$ is a hyperbolic toral group.
\end{lemma}
\begin{proof}
 Let $(v_{11}, v_{21}), (v_{12}, v_{22})\in \Bbb{R}^2$ be linearly independent vectors such that $Lat(\widetilde{Par(\Gamma_0)})=<v_1,v_2>$. Also set 
\[
\hat T
=
\left (
\begin{array}{ll}
v_{11} & v_{12}\\
v_{21} & v_{22})
\end{array}
\right );\,
T
=
\left (
\begin{array}{ll}
\hat T & 0\\
0 & 1
\end{array}
\right ).
\]
 Thus an easy calculation shows:
\[
 T^{-1}\gamma_j T
=
\left (
\begin{array}{ll}
I  & e_j\\
0 & 1
\end{array}
\right ) \textrm{ for each } j\in\{1,2\}
 \]
\[
 T^{-1}h T
=
\left (
\begin{array}{ll}
\hat T^{-1} h \hat T  & \nu_h\\
0 & 1
\end{array}
\right ) \textrm{ for each } h\in\tilde \sigma
 \]
where $\nu_h=0$ if $h=eLat(\tilde{\tau_0})$. Clearly  Proposition \ref{p:tor}, yields that 
$G=<
\{
\hat T^{-1} h:h\in\tilde \sigma
\}
>$ is a conmutative group, where each  element different from the identity is a hyperbolic toral automorphism and  whose rank is either  1 or 2. To conclude the proof, let us show  the following claim.

Claim 1.- If  there is $h\in\tilde \sigma\setminus\{\tau\}$,  then  $eLat(\tilde{\tau_0})(\nu_{eLat(h)})-\nu_{eLat(h)}\in \Bbb{Z}^2$, where $\tilde h\in \tilde \sigma$ is the lift of $h$. An easy calculation shows:
\[
\tilde \tau \tilde h\tilde \tau^{-1}\tilde h^{-1}=
\left (
\begin{array}{ll}
I & eLat(\tilde{\tau_0})(\nu_{eLat(h)})-\nu_{eLat(h)}\\
0 & 1
\end{array}
\right ).
\]
Which concludes the proof.
\end{proof}

Now theorem \ref{t:main} follows easily.

\bibliographystyle{amsplain}

\end{document}